\documentclass[12pt]{amsart}

\usepackage{geometry}
\usepackage{cite}
\usepackage[all]{xy}

\newtheorem{thm}{Theorem}[section]
\newtheorem{prop}[thm]{Proposition}
\newtheorem{lem}[thm]{Lemma}
\newtheorem{cor}[thm]{Corollary}

\theoremstyle{definition}
\newtheorem{defn}[thm]{Definition}

\theoremstyle{remark}

\newtheorem{rem}[thm]{Remark}

\numberwithin{equation}{section}

\newcommand{\bbR}{\mathbb{R}} 
\newcommand{\bbQ}{\mathbb{Q}}
\newcommand{\bbC}{\mathbb{C}}
\newcommand{\bbZ}{\mathbb{Z}}
\newcommand{\bbG}{\mathbb{G}}
\newcommand{\bbA}{\mathbb{A}}
\newcommand{\bbP}{\mathbb{P}}

\newcommand{\bbL}{\mathbb{L}}
\newcommand{\bbN}{\mathbb{N}}
\newcommand{\calG}{\mathcal{G}}

\newcommand{\calM}{\mathcal{M}}
\newcommand{\calF}{\mathcal{F}}
\newcommand{\calL}{\mathcal{L}}
\newcommand{\calO}{\mathcal{O}}

\newcommand{\calP}{\mathcal{P}}

\newcommand{\fkX}{\mathfrak{X}}
\newcommand{\lra}{\longrightarrow}
\newcommand{\lmt}{\longmapsto}
\newcommand{\conglra}{\stackrel{\cong}{\longrightarrow}}
\newcommand{\hooklongrightarrow}{\lhook\joinrel\longrightarrow}

\newcommand{\hlra}{\hooklongrightarrow}

\let \div = \relax
\DeclareMathOperator{\div}{div}

\DeclareMathOperator{\fin}{fin} 
\DeclareMathOperator{\ord}{ord} 
\DeclareMathOperator{\Div}{Div}

\DeclareMathOperator{\id}{id} 
\DeclareMathOperator{\Spec}{Spec} 
\DeclareMathOperator{\Proj}{Proj}

\DeclareMathOperator{\Sect}{Sect}

\DeclareMathOperator{\Pic}{Pic}

\DeclareMathOperator{\Cl}{Cl}
\DeclareMathOperator{\Res}{Res}
\DeclareMathOperator{\Sets}{Set}
\DeclareMathOperator{\Sch}{Sch}

\usepackage{comment}
\usepackage{amsmath}

\begin{document}

\title{Height zeta functions of projective bundles}

\author{Takuya Maruyama}
\address[Takuya Maruyama]{Graduate School of Mathematical Sciences, the University of Tokyo}
\email{maruyama@ms.u-tokyo.ac.jp}

\maketitle

\begin{abstract}
 We introduce a new approach to study height zeta functions of projective spaces and projective bundles.
 To study height zeta functions of projective spaces $Z(\bbP^n, H_{\calO(1)}; s)$, we apply the Riemann-Roch theorem of Arakelov vector bundles by van der Geer and Schoof to the integrand of an integral expression of $Z(\bbP^n, H_{\calO(1)}; s)$.
 We give a proof of the analytic continuation and functional equations of height zeta functions of projective spaces with respect to various height functions.
 Motivic analogues of these results are also proved.
 We also study height zeta functions of Hirzebruch surfaces.
%
 \end{abstract}

\section{Introduction}
\label{sect:Introduction}
Let $F$ be a number field.
Consider a projective variety $X$ over $F$
and a function $H \colon X(F) \lra \bbR$
such that the number
\[
 n(X,H;B) = \#\bigl\{P \in X(F) \bigm | H(P) \leq B \bigr\}
\]
is finite for each $B \in \bbR$.
We call a function with this property a height function on $X$.
Typically, one can construct a height function from a generically ample Arakelov line bundle on some model of $X$.
By definition, it consists of
(1) a line bundle $\calL$ on some proper model $\fkX \lra \Spec O_F$ of $X$ such that $L = \calL|_X$ is ample, and
(2) a Hermitian metric of the line bundle $L \otimes_F F_v$ on $X \otimes_F F_v$ for each complex embedding $v$ of $F$.
Then for a point $P \in X(F)$, the $1$-dimensional $F$-vector space $P^*L$ has a natural structure of an Arakelov line bundle over $\Spec O_F$.
An $F$-vector space equipped with a structure of an Arakelov vector bundle over $\Spec O_F$ is called an Arakelov vector bundle over $F$ for short.
There is a notion of the norm $N(L)$ of an Arakelov line bundle $L$ over $F$ (see Definition \ref{defn:chern class of line bundle}).
The height $H_L(P)$ of $P$ with respect to $L$ is defined as
\[
 H_L(P) = N(P^*L).
\]

Let $H_L$ be a height function on $X$ constructed as above.
We are interested in the asymptotic behaviour of the counting function $n(X,H_L;B)$ as $B \to \infty$.
It represents one aspect of the distribution of rational points of $X$ with respect to $H_L$.
A useful tool to study such problems is a Dirichlet series called the height zeta function defined by
\[
 Z(X,H_L;s) = \sum_{P \in X(F)}H_L(P)^{-s}.
\]
By a Tauberian-type theorem, one derives the asymptotic behaviour of $n(X,H_L;B)$ for $B \to \infty$ from the behaviour of $Z(X,H_L;s)$ around its abscissa of convergence.
See Theorem \ref{thm:Tauberian} for the precise statement of the Tauberian theorem used in this line.

In this paper, we study analytic properties of height zeta functions of projective spaces and Hirzebruch surfaces.
We first state our results for projective spaces.
Let $n$ be a non-negative integer and set $r = n+1$.
Let $V$ be an Arakelov vector bundle over $F$ of rank $r$.
Consider the projective space $\bbP(V)$ of all lines in $V$.
Then the line bundle $\calO_{\bbP(V)}(1)$ carries a natural structure of an Arakelov line bundle, and therefore defines a height function $H_{\calO(1)}$.
\begin{thm}
\label{thm:intro main theorem}
 In the above situation, 
 let $Z\bigl(\bbP(V),s\bigr)$ be the height zeta function
 of $\bbP(V)$ with respect to $H_{\calO(1)}$.
 \begin{enumerate}
  \item The Dirichlet series $Z\bigl(\bbP(V),s\bigr)$ becomes a holomorphic function on the domain $\bigl\{s \in \bbC \bigm| \Re(s) > r\bigr\}$, and it is analytically continued to a meromorphic function on the whole complex plane.
  \item In the domain $\{s \in \bbC \mid \Re(s) > 1\}$,
	the function $Z\bigl(\bbP(V),s\bigr)$ has an unique simple pole at $s = r$ with residue
	\[
	\Res_{s=r}Z\bigl(\bbP(V),s\bigr)=\frac{\alpha N(V)}{w|\Delta|^{\frac{r}{2}}\xi(r)}.
	\]
	Here, $N(V)$ is the norm of the Arakelov line bundle $\det(V)$, $w$ is the number of roots of unity in $F$, $\alpha = Rh$ is the product of the regulator $R$ and the class number $h$ of $F$, $\Delta$ is the discriminant of $F$, and $\xi(s)$ is the Dedekind zeta function of $F$ multiplied by a gamma factor defined in \S\ref{subs:Topologies and measures}.
  \item We have the following functional equations:
 \[
 N(V)^{-\frac{1}{2}}|\Delta|^{\frac{s}{2}}\xi(s)Z\bigl(\bbP(V),s\bigr)
 = N(V^{\vee})^{-\frac{1}{2}}|\Delta|^{\frac{r-s}{2}}\xi(r-s)Z\bigl(\bbP(V^{\vee}),r-s\bigr).
 \]
 \end{enumerate}
\end{thm}
When $V$ is a trivial Arakelov vector bundle $O_F^{\oplus r}$,
the $F$-variety $\bbP(V)$ is canonically isomorphic to $\bbP^n$,
and the height used in Theorem \ref{thm:intro main theorem} is equal to the classical height
\[
 H(P) = \prod_{v\colon\text{infinite places}}\bigl(\sum_{0\leq i \leq n}|x_i|^2_v\bigr)^{\frac{n_v}{2}}\prod_{v\colon\text{finite places}}\sup_{0 \leq i \leq n}|x_i|_v^{n_v}
\]
for a point $P \in \bbP^n(F)$ with homogeneous coordinates $[x_0:\cdots:x_n]$. 
See \S\ref{ssect:Number fields} for notation about absolute values and $n_v$'s.
In this case, the statements of Theorem \ref{thm:intro main theorem} is known (for example, this is a very special case of results in \cite{FrankeManinTschinkel} where they study height zeta functions of generalized flag varieties).
In this paper, we generalize these results to arbitrary Arakelov vector bundles using the Riemann-Roch theorem due to van der Geer and Schoof \cite{vanderGeerSchoof}.
By using this generalization, we can also handle height zeta functions of Hirzebruch surfaces.
This will be explained below.

For Hirzebruch surfaces, we prove the following.
We consider only Hirzebruch surfaces of degree greater than one.
Fix an Arakelov vector bundle $V$ of rank $2$
and let $\bbP^1$ denote the projective space $\bbP(V)$.
For an integer $e \geq 2$,
consider the Hirzebruch surface
 $\pi \colon F_e = \bbP\bigl(\calO_{\bbP^1}\oplus\calO_{\bbP^1}(e)\bigr) \lra \bbP^1$ of degree $e$.
For a pair of integers $(a,b) \in \bbZ^{2}$,
let $\calO(a,b)$ denote the line bundle $\calO_{F_e}(a)\otimes\pi^*\calO_{\bbP^1}(b)$ on $F_e$.
Here, $\calO_{F_e}(1)$ is the Serre's twisting sheaf associated to the projective bundle construction of $F_e$.
The line bundle $\calO(a,b)$ is ample if and only if $a>0$ and $b>ae$.
Assume these inequalities.
The line bundle $\calO(a,b)$ carries a natural structure of an Arakelov line bundle,
and therefore defines a height function on $F_e$.
\begin{thm}
\label{thm:intro Hirzebruch main theorem}
In the above situation, 
 let $Z(F_e,s)$ be the height zeta function of $F_e$ with respect to $H_{\calO(a,b)}$.
 Then it becomes a meromorphic function in the domain $D = \bigl\{s \in \bbC \bigm| \Re(s) > \max\{1/a,(e+2)/b\} \bigr\}$.
 There are the following two contributions to the poles of $Z(F_e,s)$ in $D$:
 \begin{enumerate}
  \item $\displaystyle \frac{\alpha Z(\bbP^1,2b/a-e)}{w |\Delta| \xi(2)}\cdot\frac{1}{as-2}\quad$ around $s = 2/a$.
  \item $\displaystyle \frac{\alpha N(V)}{w|\Delta|\xi(2)}\cdot\frac{1}{(b-ae)s-2}\quad$ around $s = 2/(b-ea)$ if it is in $D$.
 \end{enumerate} 
\end{thm}
In particular,
there is a decomposition of the ample cone $\Lambda_{\text{ample}}(F_e)$ into two cones such that 
the asymptotic behaviour of the counting function $n(F_e,H_{\calO(a,b)};B)$ varies on which cone the line bundle $\calO(a,b)$ belongs to.
For more discussions on this result, see \S\ref{ssect:Discussions}.

In \S\ref{sect:Motivic case}, we treat other versions of height zeta functions.
To define these functions,
fix a proper smooth curve $C$ over a field $k$.
Consider a proper scheme $X \lra C$
and a relatively ample line bundle $L$ on $X$.
Then for a section $P \in X(C)$, 
we define the logarithmic height $h_L(P)$ as
\[
 h_L(P) = \deg (P^*L).
\]

If $k$ is finite, then the number of points $P \in X(C)$ with bounded heights is finite.
Therefore we can define the geometric height zeta function of $X$ with respect to $h_L$ as the power series
\[
 Z_{\text{geom}}(X,h_L;t) = \sum_{P\in X(F)}t^{h_L(P)} \in \bbZ((t)).
\]
For a general field $k$, we can define the motivic height zeta function
$Z_{\text{mot}}(X,h_L;t)$ of $X$
as a power series with coefficients in the Grothendieck ring $K$ of varieties over $k$.
When $k$ is finite, there is a ring homomorphism $\mu \colon K \lra \bbZ$ such that $\mu\bigl([X]\bigr) = \#X(k)$ for any $k$-variety $X$.
We also write $\mu$ for the ring homomorphism $K((t)) \lra \bbZ((t))$ induced by $\mu$.
Then we have $\mu\bigl(Z_{\text{mot}}(X,h_L;t)\bigr) = Z_{\text{geom}}(X,h_L;t)$.

In \cite{Wan}, Daqing Wan studies the geometric height zeta function
of trivial projective bundles $X = \bbP^n \times C \lra C$ with respect to $L = \calO_{X/C}(1)$.
For example, he proves that $Z_{\text{geom}}\bigl(\bbP^n \times C,h_{\calO(1)};t\bigr)$ is rational in $t$ and describes its denominator using the Riemann-Roch theorem.
In this paper, we study motivic height zeta functions
when $X \lra C$ is a projective bundle which do not necessarily trivial.
To state our results, let $\widetilde{K}$ be the quotient of $K$ by an ideal generated by $[X]-[Y]$ for all radicial surjective morphisms $f \colon X \lra Y$.
Let $\bbL$ denote the image of $\bbA^1$ in $\widetilde{K}$, and let $\calM = \widetilde{K}[\bbL^{-1}]$ be the localization.
 Let $\zeta(t) \in K[[t]]$ denote the Kapranov's motivic zeta function of $C$.
 It is known that $\zeta(t)$ is a rational function in $\calM(t)$.
\begin{thm}
\label{thm:intro motivic main theorem}
 Let $V$ be a vector bundle of rank $r$ over $C$.
 We denote the motivic height zeta function of $\bbP(V) \lra C$ with respect to $\calL = \calO_{\bbP(V)}(1)$
 by $Z\bigl(\bbP(V),t\bigr)$.
 \begin{enumerate}
  \item The product $\zeta(t)Z\bigl(\bbP(V),t\bigr)$ is a rational function in $\calM(t)$ with the denominator $(t-1)(t-\bbL^{-r})$.
	Since $\zeta(t)$ is rational in $\calM((t))$, we also have the rationality of $Z\bigl(\bbP(V),t\bigr)$ in $\calM((t))$.
  \item The value of $(t-1)(t-\bbL^{-r})\zeta(t)Z\bigl(\bbP(V),t\bigr)$ at $t = \bbL^{-r}$
	is $\bbL^{-rg-1+\deg V}\bigl(1-[\bbP^{-r}]\bigr)[J]$.
	Here, $J$ denotes the Jacobian variety of $C$,
	and the negative dimensional projective space $[\bbP^{-r}]$ is defined by
	\[
	 [\bbP^{-r}] = -\bbL^{-(r-1)} - \bbL^{-(r-2)} - \cdots - \bbL^{-1}.
	\]
  \item We have the following functional equations:
	\[
	\zeta(t)Z\bigl(\bbP(V),t\bigr) = \bbL^{\deg V+r(g-1)}t^{2g-2}\zeta(\bbL^{-r}t^{-1})Z\bigl(\bbP(V^{\vee}),\bbL^{-r}t^{-1}\bigr).
	\]
 \end{enumerate}
\end{thm}

This is a precise motivic analogue of Theorem \ref{thm:intro main theorem}.

Our approach in the proof of Theorem \ref{thm:intro main theorem} is an Arakelov theoretic analogue of (a refined version of) Wan's arguments in \cite{Wan}.
The Riemann-Roch theorem in his arguments is repleced by van der Geer and Schoof's version of Tate's arithmetic Riemann-Roch theorem given in \cite{vanderGeerSchoof}.
In \S2, we recall van der Geer and Schoof's theory of \cite{vanderGeerSchoof}, which also plays an important role in the proof of \ref{thm:intro Hirzebruch main theorem}.

\section{Arakelov theory}
In this section, we introduce notiation in Arakelov theory of arithmetic curves and recall results from \cite{vanderGeerSchoof}.

\subsection{Number fields}\label{ssect:Number fields}
Throughout in this paper, we fix a number field $F$ and use the following notation:
$O_F$ is the ring of integers,
$\mu_F$ is the group of roots of unity,
$w = \#\mu_F$, 
$\Delta$ is the discriminant,
$R$ is the regulator, $h$ is the class number,
$\alpha = Rh$, and
$\zeta(s)$ is the Dedekind zeta function.
The symbol $v$ always denotes a finite or infinite place of $F$.
The number of real (resp. complex) places is denoted as $r_1$ (resp. $r_2$).
For a place $v$, the function $|\cdot|_v \colon F \lra \bbR$ is the absolute value of $F$ which extends the standard $p$-adic or archimedean absolute value on $\bbQ$.
$F_v$ is the completion of $F$ along $|\cdot|_v$, and
$n_v$ is the local degree at $v$ (i.e., $[F_v:\bbQ_w]$ for a place $w$ of $\bbQ$ such that $v|w$).
For a finite place $v$, the norm $N(v)$ is the cardinality of the residue field at $v$.

\subsection{Arakelov bundles and divisors}
\begin{defn}
 \begin{enumerate}
  \item A Hermitian module is a finitely generated projective $O_F$-module $M$ equipped with a Hermitian metric $\langle,\rangle_{\sigma}$ on $M \otimes_{\sigma} \bbC$ for each embedding $\sigma \colon F \hlra \bbC$.
	A Hermitian module $M$ is called of real type if it satisfies $\overline{\langle x \otimes_{\sigma} 1, y \otimes_{\sigma} 1 \rangle_{\sigma}} = \langle x \otimes_{\overline{\sigma}} 1, y \otimes_{\overline{\sigma}} 1 \rangle_{\overline{\sigma}}$ for each $x,y \in M$ and $\sigma$.
  \item An Arakelov vector bundle $V$ is a vector space over $F$,
	equipped with a $O_F$-lattice $\Gamma(V) \subset V$ which is a Hermitian module of real type.
  \item An Arakelov vector bundle of rank $1$ is called an Arakelov line bundle.
	The Arakelov Picard group $\Pic(F)$ is defined as the group of isometry classes of Arakelov line bundles.
  \item Let $V$ be an Arakelov vector bundle.
	Choose an embedding $\sigma_v \colon F_v \hlra \bbC$ for each place $v|\infty$.
	Then the Hermitian metric defines a $F_v$-norm $\|\cdot\|_v$ on $V \otimes_F F_v$ by 
	$\|x\|_v^2 = \langle x \otimes_{\sigma_v} 1,x \otimes_{\sigma_v} 1\rangle_{\sigma_v}$.
	This does not depend on the choice of $\sigma_v$ because $\Gamma(V)$ is of real type.
 \end{enumerate}
\end{defn}

Various constructions on finite dimensional vector spaces, such as direct sums, duals, tensor products and determinants, are also defined for Arakelov bundles.
We call an element of the underlying vector space $V$ a rational section.
\begin{defn}
 \begin{enumerate}
  \item An Arakelov divisor is a formal finite sum $\sum_{v|\infty}x_v[v] + \sum_{v\colon\fin}x_v[v]$, where $x_v \in \bbR$ for $v|\infty$ and $x_v \in \bbZ$ for finite $v$.
	We denote the group of Arakelov divisors by $\Div(F)$.
  \item For a rational function $f \in F^{\times}$, we define the principal divisor associated to $f$ as
	\[
	 (f) = \sum_{v|\infty}\bigl(-\log|f|^{n_v}_v\bigr)[v] + \sum_{v:\fin}\ord_v(f)[v].
	\]
	The Arakelov class group $\Cl(F)$ is the quotient of $\Div(F)$ by the group of principal divisors.
  \item The degree of an Arakelov divisor $D$ is defined as
	\[
	\deg(D) = \sum_{v|\infty}x_v + \sum_{v:\fin}\log N(v) x_v.
	\]
	The norm of $D$ is defined as $N(D) = \exp\bigl(\deg(D)\bigr)$.
 \end{enumerate}
\end{defn}
Note that the degree and the norm are well-defined for a class in $\Cl(F)$ by the product formula.

\begin{defn}
\label{defn:chern class of line bundle}
 \begin{enumerate}
  \item For an Arakelov line bundle $L$ and a rational section $x$ of $L$, we define the divisor of zeros and poles of $x$ as
	\[
	\div(x) = \sum_{v|\infty}(-\log \|x\|^{n_v}_v)[v] + \sum_{v:\fin}\ord_v(x)[v].
	\]
	Then the class $[\div(x)] \in \Cl(F)$ does not depend on $x$, which we denote as $c_1(L)$.
  \item For an Arakelov divisor $D$, we define an Arakelov line bundle $O_F(D)$ as a line bundle such that the vector space of rational sections is $F$ and that $\div(1) = D$.
  \item We define $\deg L = \deg c_1(L)$.
	We also define the degree of an Arakelov vector bundle $V$ as
	$\deg(V) = \deg \bigl(\det(V)\bigr)$.
	We define the norm of $V$ as $N(V) = \exp\bigl(\deg(V)\bigr)$.
 \end{enumerate}
\end{defn}
It is easily checked that the map $L \lmt c_1(L)$ is a bijection from the Picard group $\Pic(F)$ to the class group $\Cl(F)$.

\subsection{The Riemann-Roch and a vanishing result}\label{subs:Riemann-Roch}
The notion of Arakelov vector bundles is an arithmetic analogue of vector bundles on algebraic curves over finite fields.
The underlying $F$-vector space $V$ is considered as the space of rational sections, and the module $\Gamma(V) \subset V$ is considered as the set of rational sections which are regular at finite places.

One usual definition of ``global sections'' is as follows.
For an infinite place $v$, a rational section $x \in V$ is defined to be ``regular at $v$'' if and only if $\|x\|_v \leq 1$.
So one defines the set of global sections $H^0(V)$ as $\bigl\{x \in \Gamma(V)\bigm| \|x\|_v \leq 1 \text{ for all }v|\infty\bigr\}$,
and the ``dimension'' $h^0(V)$ of $H^0(V)$ as $\log \# H^0(V)$.

Another and beautiful idea of a definition of $h^0(V)$ is given in \cite{vanderGeerSchoof}.
For a place $v|\infty$ and a section $x \in \Gamma(V)$,
we consider the quantity $\exp(-n_v\pi\|x\|_v^2)$ as ``the probability of $x$ to be regular at $v$''.
Then the probability of $x$ to be a global section is equal to 
$\prod_{v|\infty}\exp(-n_v\pi\|x\|_v^2)$.
Therefore we make the following definitions.

\begin{defn}
\label{defn:beautiful definitions}
 Let $V$ be an Arakelov bundle.
 \begin{enumerate}
  \item For an element $x \in \Gamma(V)$, we define a positive real number $e_V(x)$ by
	\[
	 e_V(x) = \prod_{v|\infty}\exp(-n_v\pi\|x\|_v^2).
	\]
  \item We define ``the number of global sections'' $\#H^0(V)$ by
	\[
	\#H^0(V) = \sum_{x \in \Gamma(V)}e_V(x)
	\]
	and ``the dimension of $H^0(V)$'' by $h^0(V) = \log \#H^0(V)$.
	We also use $\varphi(V) = \#H^0(V) - 1$, which is considered as ``the number of nonzero global sections''.
 \end{enumerate}
\end{defn}

Now we state the Riemann-Roch theorem and a vanishing result of Arakelov vector bundles.
The canonical bundle $\omega$ is defined to be an Arakelov bundle such that $\Gamma(\omega)$ is the inverse of the different ideal of $F$ and $\|1\|_v = 1$ for all infinite places $v$.
Then $\deg \omega = \log |\Delta|$.

\begin{prop}\label{prop:R-R and vanishing}
 Let $V$ be an Arakelov vector bundle of rank $r$. 
 \begin{enumerate}
  \item (Riemann-Roch) $h^0(V)-h^0(V^{\vee}\otimes\omega)=\deg V - \frac{r}{2}\log|\Delta|$.
  \item (Vanishing)
	For any $C \in \bbR$, 
	there are constants $C_1,C_2 > 0$ such that
	\[
	\varphi(V\otimes L) \leq C_1 \exp \Bigl(-C_2 \exp \Bigl(-\frac{2}{[F:\bbQ]} \deg L\Bigr) \Bigr)
	\]
	for all $L \in Pic(F)$ with $\deg L \leq C$.
	In other words, the function $L \lmt \varphi(V \otimes L)$ on $\Pic(F)$ tends to zero doubly exponentially fast and uniformly as $\deg L \to -\infty$.
  \item There is a constant $C > 0$ such that $\varphi(L) \leq CN(L)$ for any $L \in \Pic(F)$.
 \end{enumerate}
\end{prop}

\begin{proof}
 As explained in \cite{vanderGeerSchoof} (where they prove (i) when $r=1$), the proof of (i) is an easy applicatoin of the Poisson summation formula.
 
 Since $\Pic^0(F)$ is compact (see \S\ref{subs:Topologies and measures}) and $L \lmt \varphi(V\otimes L)$ is continuous, it suffices to show the statement (ii) for one choice of $C$.
 Any Arakelov vector bundle $V$ is embedded into a direct sum $\oplus_{1 \leq i \leq n}L_i$ of Arakelov line bundles $L_i$.
 Then the proof of (ii) is reduced to the case $r=1$ and $C = \frac{1}{2}\log|\Delta|$, which is proved in \cite{vanderGeerSchoof}.

 The statement (iii) follows from (i), (ii) and the compactness of $\Pic^0(F)$.
\end{proof}

\subsection{Topologies and measures}\label{subs:Topologies and measures}
We endow the group $\Div(F) = \prod_{v|\infty}\bbR \times \bigoplus_{v:\fin}\bbZ$ with a natural topology and a Haar measure, i.e., the product of Euclidean and discrete topologies, and the product of Lebesgue and counting measures.
Then the group $\Pic(F) \cong \Cl(F) = \Div(F) / (F^{\times}/\mu_F)$
is endowed with the quotient topology and the quotient measure.
Let $\Pic^0(F) \subset \Pic(F)$ be the kernel of the degree map:
\[
 0 \lra \Pic^0(F) \lra \Pic(F) \stackrel{\deg}{\lra} \bbR \lra 0.
\]
The group $\Pic^0(F)$ fits into the exact sequence
\[
 0 \lra H/\phi(O_F^{\times}) \lra \Pic^0(F) \lra \Pic(O_F) \lra 0,
\]
where $\Pic(O_F)$ is the group of isomorphism classes of finitely generated projective $O_F$-modules of rank $1$, $H =\bigl \{ (x_v) \in \prod_{v|\infty}\bbR \bigm | \sum_{v} x_v = 0 \bigr \}$, and $\phi(f) = \bigl( \log|f|_v^{n_v} \bigr)_v$.
Then Dirichlet's unit theorem (i.e., the compactness of $H/\phi(O_F^{\times})$) and the finiteness of ideal class group (i.e., the finiteness of $\Pic(O_F)$) implies that $\Pic^0(F)$ is compact.
The volume of $\deg^{-1}\bigl([0,1]\bigr) \subset \Pic(F)$ is equal to $\alpha = Rh$.
For a function $f \in L^1(\bbR)$ 
we have $\int_{\Pic(F)}f(\deg L)dL = \alpha\int_{\bbR}f(t)dt$.

The completed Dedekind zeta function of $F$ can be expressed as an integral over $\Div(F)$.
To explain this, we define the effectivity $e(D)$ (i.e., ``the probability of $D$ to be effective'') of an Arakelov divisor $D$.
\begin{defn}
 For an Arakelov divisor $D$, we define
 \[
  e(D) = \begin{cases}
	  0 & \text{if the finite part of }D\text{ is effective, and}\\
	  e_{O_F}(1) & \text{otherwise.}
	 \end{cases}
 \]
 Here, $e_{O_F}(1)$ denotes ``the probability of $1$ to be a global section of $O_F(1)$'' defined in Definition \ref{defn:beautiful definitions}.

 Explicitly, if the finite component of $D$ is effective and the infinite component is $\sum_{v|\infty}x_v[v]$, then the effectivity $e(D)$ is given by $\prod_{v|\infty}\exp(-n_v\pi\exp(-\frac{2}{n_v}x_v))$.
\end{defn}

Now we define a function $\xi(s)$ as an integral
\[
 \xi(s) = \int_{\Div(F)}N(D)^{-s}e(D)dD.
\]
By the decomposition $\Div(F) = \bigoplus_{v\colon\fin}\bbZ \times \prod_{v|\infty}\bbR$, we have
\begin{align*}
 \xi(s) &= \sum_{I \subset \calO_F} N(I)^{-s} \prod_{v|\infty}\int_{\bbR}e^{-xs}e^{-n_v\pi\exp(-\frac{2}{n_v}x)}dx \\
 &= 2^{-r_1}\Bigl(\pi^{-\frac{s}{2}}\Gamma\Bigl(\frac{s}{2}\Bigr)\Bigr)^{r_1}\bigl((2\pi)^{-s}\Gamma(s)\bigr)^{r_2}\zeta(s).
\end{align*}
(The computation in \cite{vanderGeerSchoof} is incorrect by a scalar multiple.)

We can also express $\xi(s)$ as an integral over $\Pic(F)$.
Note that the effectivity $e\bigl(D+(f)\bigr)$
is equal to $e_{O(D)}(f)$.
Therefore,
\[
 \xi(s) = \int_{\Pic(F)}N\bigl([D]\bigr)^{-s}\sum_{f \in F^{\times}/\mu_F}e\bigl(D+(f)\bigr)d[D] = w^{-1} \int_{\Pic(F)}N(L)^{-s}\varphi(L)dL.
\]

\section{Projective spaces}
\label{sec:Height zeta functions}
In this section, we prove the analytic continuation and functional equations of height zeta functions of projective spaces.
\subsection{Heights}
\label{ssec:Height}
Let $V$ be an Arakelov vector bundle of rank $r = n+1$,
and $\bbP(V)$ be the associated projective space of lines in $V$.
A $1$-dimensional subspace $L \subset V$
corresponds to a rank $1$ projective $O_F$-submodule $\Gamma(V) \cap L$ of $\Gamma(V)$ with a projective cokernel.
So we have an Arakelov subbundle of $V$ of rank $1$ for each $F$-rational point $P$ of $\bbP(V)$, which is denoted as $P^*\calO(-1)$.
The dual of $P^*\calO(-1)$ is denoted as $P^*\calO(1)$.
We define the height $H(P)$ of $P$ as the norm $N\bigl(P^*\calO(1)\bigr)$.

\subsection{A generalization of Wan's formula}
The height zeta function is defined as the series
\begin{align}\label{eq:def of height zeta}
 Z\bigl(\bbP(V),s\bigr) = \sum_{P \in \bbP(V)(F)}H(P)^{-s} \quad\text{for }s \in \bbC.
\end{align}
We are going to show that the series $Z\bigl(\bbP(V),s\bigr)$ converges absolutely for $\Re s > r$, and it is analytically continued to a meromorphic function on the whole complex plane.

Let $\xi(s)$ be the function defined in \S\ref{subs:Topologies and measures}.
Then for a point $P \in \bbP(V)(F)$, we have
\begin{align*}
 w\xi(s)H(P)^{-s} &= \int_{\Pic(F)}N\bigl(L\otimes P^*\calO(1)\bigr)^{-s} \varphi(L)dL \\
 &= \int_{\Pic(F)}N(L)^{-s} \varphi\bigl(L\otimes P^*\calO(-1)\bigr)dL.
\end{align*}
If we fix $L \in \Pic(F)$ and move $P$ in the set $\bbP(V)(F)$, 
then the line bundle $L \otimes P^*\calO(-1)$ runs through
all of line-subbundles of $L \otimes V$.
Note that, for a rational section $x \in L \otimes P^*\calO(-1)$,
the quantity $e_{L \otimes P^*\calO(-1)}(x)$ is equal to $e_{L \otimes V}(x)$.
Therefore, by the definition of the height zeta function (\ref{eq:def of height zeta}), we have
\begin{align}
 \begin{aligned}\label{eq:Mobius inversion}
 w \xi(s)Z\bigl(\bbP(V),s\bigr)
 &= \int_{\Pic(F)}N(L)^{-s}\sum_{P \in \bbP(V)(F)}\varphi\bigl(L \otimes P^*\calO(-1)\bigr)dL \\
 &= \int_{\Pic(F)}N(L)^{-s}\varphi(L \otimes V)dL. 
\end{aligned}
\end{align}

\begin{rem}
 Consider a new function $\xi^{(k)}(s)$ defined by
 \[
 \xi^{(k)}(s) = \int_{\Div(F)}\biggl(\frac{\varphi\bigl(O_F(D)\bigr)}{w}\biggr)^k N(D)^{-s}e(D)dD = w^{-(k+1)}\int_{\Pic(F)}\varphi(L)^{k+1}N(L)^{-s}dL
 \]
 for $k \geq 0$.
 This is an arithmetic analogue of the $k$-th zeta function in \cite{Wan}.
 
 Suppose that $V=F^{\oplus r}$ is a trivial bundle of rank $r$.
 In this case, since $\varphi(L^{\oplus r})+1 = \bigl(\varphi(L)+1\bigr)^r = \sum_{0 \leq l \leq r}\binom{r}{l}\varphi(L)^l$, the formula (\ref{eq:Mobius inversion}) implies
 \[
  Z\bigl(\bbP^n,s\bigr) = \frac{1}{\xi(s)}\sum_{0 \leq k \leq n}\binom{n+1}{k+1}w^{k}\xi^{(k)}(s).
 \]
 This is our version of Wan's formula in Theorem 3.1 of \cite{Wan}.
\end{rem}

\subsection{An Application of the Riemann-Roch}
\label{ssect:Application of Riemann-Roch}
Let $E_{+} = \{L \in \Pic(F) \mid N(L) \geq \sqrt{|\Delta|}\}$ and $E_{-} = \{L \in \Pic(F) \mid N(L) \leq \sqrt{|\Delta|}\}$.
Divide the right hand side of (\ref{eq:Mobius inversion}) into the sum of two integrals:
\[
 \int_{\Pic(F)}N(L)^{-s}\varphi(L \otimes V)dL = F_+(V,s) + F_-(V,s)
\]
where
\[
 F_{\pm}(V,s) = \int_{E_{\pm}}N(L)^{-s}\varphi(L\otimes V)dL.
\]
By (ii) of Proposition \ref{prop:R-R and vanishing}, the integral $F_{-}(V,s)$ converges to a holomorphic function on the whole complex plane.
As for the integral $F_{+}(V,s)$, by substituting $L^{\vee} \otimes \omega$ for $L$,
we have the following expression over $E_{-}$:
\[
 F_{+}(V,s) = \int_{E_{-}}|\Delta|^{-s}N(L)^s\varphi(L^{\vee} \otimes V \otimes \omega)dL.
\]
By the Riemann-Roch theorem ((i) of Proposition \ref{prop:R-R and vanishing}) applied to $L \otimes V^\vee$, we have
\begin{align*}
 \varphi(L^{\vee} \otimes V \otimes \omega) + 1 &= |\Delta|^{\frac{r}{2}}N(L^{\vee} \otimes V) \bigl(\varphi(L \otimes V^{\vee})+1\bigr) \\
 \iff \varphi(L^{\vee} \otimes V \otimes \omega)
 &= |\Delta|^{\frac{r}{2}}N(V)N(L)^{-r} \bigl(\varphi(L \otimes V^{\vee})+1\bigr) - 1.
\end{align*}
Theorefore,
\begin{align*}
 N(V)^{-\frac{1}{2}}|\Delta|^{\frac{s}{2}}F_{+}(V,s)
 = \int_{E_{-}}N(V)^{\frac{1}{2}}|\Delta|^{\frac{r-s}{2}}N(L)^{s-r}\varphi(L \otimes V^{\vee})dL \\
 + \int_{E_{-}}N(V)^{\frac{1}{2}}|\Delta|^{\frac{r-s}{2}}N(L)^{s-r}dL
 - \int_{E_{-}}N(V)^{-\frac{1}{2}}|\Delta|^{-\frac{s}{2}}N(L)^{s}dL.
\end{align*}
The first term in the right hand side converges to a holomorphic function on the whole complex plane by (ii) of Proposition \ref{prop:R-R and vanishing}.
Moreover, by the formula $\int_{\Pic(F)}f(\deg L)dL = \alpha \int_{\bbR}f(t)dt$ in \S\ref{subs:Topologies and measures}, the second and third terms are computed explicitly as follows:
\begin{align*}
 \int_{E_{-}}N(V)^{\frac{1}{2}}|\Delta|^{\frac{r-s}{2}}N(L)^{s-r}dL &=
 \frac{N(V)^{\frac{1}{2}}\alpha}{s-r} \quad (\Re(s) > r),\\
 \int_{E_{-}}N(V)^{-\frac{1}{2}}|\Delta|^{-\frac{s}{2}}N(L)^{s}dL &=
 \frac{N(V)^{-\frac{1}{2}}\alpha}{s} \quad (\Re(s) > 0).
\end{align*}
Summarizing, we obtained the following expression of the right hand side of (\ref{eq:Mobius inversion}) multiplied by $N(V)^{-\frac{1}{2}}|\Delta|^{\frac{s}{2}}$ which is valid for $s$ with $\Re(s) > r$:
 \begin{align}
  \begin{aligned}\label{eq:star}
   N(V)^{-\frac{1}{2}}|\Delta|^\frac{s}{2}\bigl(F_{-}(V,s) + F_{+}(V,s)\bigr)
 &= \int_{E_{-}}N(V)^{-\frac{1}{2}}|\Delta|^{\frac{s}{2}}N(L)^{-s}\varphi(L\otimes V)dL \\
 &+ \int_{E_{-}}N(V)^{\frac{1}{2}}|\Delta|^{\frac{r-s}{2}}N(L)^{s-r}\varphi(L \otimes V^{\vee})dL \\
 &+ \frac{N(V)^{\frac{1}{2}}\alpha}{s-r}
 -\frac{N(V)^{-\frac{1}{2}}\alpha}{s}.
  \end{aligned}
 \end{align}
Note that the right hand side is invariant under the replacement $(V,s) \leftrightarrow (V^{\vee},r-s)$.
Now we have proved the following theorem.
\begin{thm}
 \label{thm:main theorem}
 Let $V$ be an Arakelov vector bundle of rank $r = n + 1$.
 Then the height zeta function $Z\bigl(\bbP(V),s\bigr)$ converges to a holomorphic function on the domain $\{s \in \bbC \mid \Re(s) > r\}$, and it is analytically continued to a meromorphic function on the whole complex plane.
 In the domain $\{s \in \bbC \mid \Re(s) > 1\}$, it has an unique simple pole at $s = r$ with residue
 \[
  \Res_{s=r}Z\bigl(\bbP(V),s\bigr)=\frac{\alpha N(V)}{w|\Delta|^{\frac{r}{2}}\xi(r)}.
 \]

 Moreover, these functions satisfy the following functional equations:
 \[
 N(V)^{-\frac{1}{2}}|\Delta|^{\frac{s}{2}}\xi(s)Z\bigl(\bbP(V),s\bigr)
 = N(V^{\vee})^{-\frac{1}{2}}|\Delta|^{\frac{r-s}{2}}\xi(r-s)Z\bigl(\bbP(V^{\vee}),r-s\bigr).
 \]
 \end{thm}


\section{Motivic case}
\label{sect:Motivic case}
In this section, we prove the analytic continuation and functional equations of motivic height zeta functions of projective bundles.

\subsection{Zeta functions}
Fix a field $k$ and let $K$ be the Grothendieck ring of quasi-projective varieties over $k$.
So $K$ is a ring generated by isomorphism classes of varieties over $k$ with multiplication defined by products of varieties, 
modulo relations of the form $[X] = [Y] + [X \setminus Y]$ for each closed subvariety $Y$ of $X$.
Consider an ideal $I \subset K$ generated by elements of the form $[X]-[Y]$ for each radicial surjective morphism $f \colon X \lra Y$.
Let $\widetilde{K} = K/I$.
We denote the image of $\bbA^1$ in $\widetilde{K}$ by $\bbL$.
In the last, let $\calM = \widetilde{K}[\bbL^{-1}]$ be the localization of $\widetilde{K}$ by $\bbL$.

Let $C$ be a smooth projective curve over $k$.
For simplicity we assume that $C$ has a $k$-rational point $p_0$.
Kapranov's motivic zeta function of $C$ is defined as a series in $K[[t]]$ whose $n$-the coefficient is the symmetric product $C_n$ of $C$:
\[
 \zeta(t) = \sum_{n \geq 0}[C_n]t^n.
\]
Note that $C_0 = \Spec k$, so the constant term of $\zeta(t)$ is $1$.
For later convenience, we define $C_n = \emptyset$ for negative $n$.
Kapranov has shown that $\zeta(t)$ is a ratioanl function in $\calM(t)$ of the form
\[
 \zeta(t) = \frac{f(t)}{(1-t)(1-\bbL t)},
\]
where $f(t)$ is a polynomial in $\calM[t]$ of degree $\leq 2g$.
Moreover, the motivic zeta function of $C$ satisfies the following functional equation
\[
 \zeta(t) = \bbL^{g-1}t^{2g-2}\zeta\bigl(\bbL^{-1}t^{-1}\bigr)
\]
in $\calM(t)$.

Let $V$ be a vector bundle over $C$ of rank $r$,
and $\pi \colon \bbP(V) \lra C$ be the associated projective bundle of line-subbundles of $V$.
Let $\calO(-1)$ denote the universal subbundle of $\pi^*V$, and $\calO(1)$ the dual.
For each $d \in \bbZ$, let $\Sect_d\bigl(\bbP(V)\bigr)$ be the scheme of sections of $\bbP(V) \lra C$ of degree $d$.
By definition, it is a scheme which represents the functor
\[
\Sch_k \lra \Sets; T \lmt \bigl\{\text{sections } s \text{ of }\pi \times \id_T \colon \bbP(V) \times T \lra C \times T \bigm | \forall t \in T, \deg\bigl(s_t^*\calO(1)\bigr) = d\bigr\}.
\]
The above functor is represented by an open subscheme of the Hilbert scheme of closed subschemes of $\bbP(V)$ of dimension $1$ and degree $d$ with respect to $\calO(1)$.
In particular, $\Sect_d\bigl(\bbP(V)\bigr)$ is a quasi-projective variety.

\begin{lem}
\label{lem:Sect_d empty for small d}
 For enough small $d$, $\Sect_d\bigl(\bbP(V)\bigr)$ is an empty set.
\end{lem}

\begin{proof}
 Fix an ample line bundle $L$ on $C$.
 There exists an integer $m$ such that $V^{\vee} \otimes L^{\otimes m}$ is ample.
 Then for any field $k'$ containing $k$ and a section $s \colon C_{k'} \lra \bbP(V_{k'})$ of degree $d$,
 the line bundle $s^*\bigl(\calO_{\bbP(V\otimes L^{\otimes(-m)})}(1)\bigr) = s^*\calO_{\bbP(V)}(1) \otimes L^{\otimes m}$ is ample,
 and therefore its degree $d + m$ is positive.
 This implies that $\Sect_d\bigl(\bbP(V)\bigr) = \emptyset$ if $d \leq - m$.
\end{proof}

Now we define the motivic height zeta function of $\bbP(V)$ as the power series
\[
 Z\bigl(\bbP(V),t\bigr) = \sum_{d \in \bbZ}\bigl[\Sect_d\bigl(\bbP(V)\bigr)\bigr]t^d,
\]
which belongs to $K((t))$ by Lemma \ref{lem:Sect_d empty for small d}.

\subsection{The variety $X_n(V)$}
Consider the product of zeta function of $C$ and height zeta function:
\[
 \zeta(t)Z\bigl(\bbP(V),t\bigr) = \sum_{n \in \bbZ} \sum_{d \in \bbZ} \bigl[\Sect_d\bigl(\bbP(V)\bigr)\bigr][C_{n-d}]t^n.
\]
In this subsection, we define a variety $X_n(V)$ and prove the equality
$\sum_{d \in \bbZ} \bigl[\Sect_d\bigl(\bbP(V)\bigr)\bigr][C_{n-d}] = \bigl[X_n(V)\bigr]$ in $\widetilde{K}$.

For a coherent sheaf $\calF$, let $\bbP^{\vee}(\calF)$ denote the $\Proj$-construction applied to the symmetric algebra of $\calF$.
For a surjective morphism $\calF \lra \calG$ of coherent sheaves there is an associated closed embedding $\bbP^{\vee}(\calG) \lra \bbP^{\vee}(\calF)$.
If $\calF$ is the sheaf of sections of a vector bundle $V$, then
$\bbP^{\vee}(\calF^{\vee})$ is our projective bundle $\bbP(V)$ of line-subbundles of $V$.

Fix a point $p_0 \in C(k)$.
Let $J_n = \Pic^n(C)$ be the Picard scheme of line bundles of degree $n$.
We write simply $J$ for $J_0$.
Let $\calP_n$ be the Poincar\'e sheaf on $C \times J_n$ normalized such that $\calP_n|_{\{p_0\} \times J_n} \cong \calO_{J_n}$ and $\calP_n|_{C \times \{\calO(np_0)\}} \cong \calO_C(np_0)$.
Schwarzenberger \cite{Schwarzenberger} showed that the natural morphism $C_n \lra J_n$ is isomorphic to $\bbP^{\vee}\bigl(R^1p_{2*}(\calP_n^{\vee} \otimes p_1^*\omega_C)\bigr) \lra J_n$ (where $p_i$ ($i = 1,2$) denote projection morphisms from $C \times J_n$ to $C$, or $J_n$).
We define $X_n(V) = \bbP^{\vee}\bigl(R^1p_{2*}(\calP_n^{\vee} \otimes p_1^*V^{\vee} \otimes p_1^*\omega_C)\bigr)$ (here we simply write $V$ for the sheaf of sections of $V$).
The following lemma is a generalization of Proposition 7 and 8 in \cite{Schwarzenberger}.
\begin{lem}
\label{lem:coh and base chg}
 For any scheme $S$ and morphism $h\colon S \lra J_n$,
 there is an isomorphism $X_n(V) \times_{J_n} S \cong \bbP^{\vee}\bigl(R^1p_{2*}((\id_C \times h)^*(\calP_n^{\vee} \otimes p_1^{*}V^{\vee} \otimes p_1^*\omega_C))\bigr)$.
 In particular, if $S = \Spec k'$ for a field $k'$
 and $h$ corresponds to a sheaf $\calL \in \Pic^n(C_{k'})$,
 then the fiber of $X_n(V) \lra J_n$ at $\calL \in J_n(k')$ is isomorphic to
 $\bbP^{\vee}\bigl(H^1(C_{k'},\calL^{\vee} \otimes V^{\vee}_{k'} \otimes \omega_{C_{k'}})\bigr) \cong \bbP\bigl(H^0(C_{k'},\calL \otimes V_{k'})\bigr)$.
\end{lem}
\begin{proof}
 Schwarzenberger proves this lemma when $V = \calO_C$, and his proof also applies to our case.
 Here we give a shorter proof.
 Since top higher direct images always commute with base change,
 there is a canonical isomorphism
 \[
  h^*R^1p_{2*}\calG \conglra R^1p_{2*}\bigl((\id_C \times h)^*\calG\bigr)
 \]
 for any coherent sheaf $\calG$ on $C \times J_n$ and $n \in \bbZ$.
 The construction $\bbP^{\vee}$ is compatible with pull-back of coherent sheaves.
 Now the lemma is clear.
 \end{proof}
\label{prop:computation of prod of sect and C}
\begin{prop}
 \begin{enumerate}
  \item There is a natural morphism $\Sect_d\bigl(\bbP(V)\bigr) \times C_{n-d} \lra X_n(V)$.
  \item The morphism $m\colon\coprod_{d \in \bbZ}\Sect_d\bigl(\bbP(V)\bigr) \times C_{n-d} \lra X_n(V)$ induced by (i) is radicial and surjective.	
 \end{enumerate}
\end{prop}

\begin{proof}
 We recommend the reader to look at the proof of (ii) at first
 to know the geometric description of the morphism in (i)
 and then back to the proof of (i) for the scheme-theoretic construction of the morphism.
 
 (i) We construct a morphism between functors represented by $\Sect_d\bigl(\bbP(V)\bigr) \times C_{n-d}$ and $X_n(V)$.
 Let $T$ be a $k$-scheme and pick an element $(s,f)$ of $\Sect_d\bigl(\bbP(V)\bigr)(T) \times C_{n-d}(T)$.
 Let $\overline{f}$ denote the composition $T \stackrel{f}{\lra} C_{n-d} \lra J_{n-d}$, and $\calO(f)$ denote the sheaf $(\id_C \times \overline{f})^*\calP_{n-d}$ on $C \times T$.
 The point $s$ corresponds to a $C$-morphism $C \times T \lra \bbP(V)$,
 which is also denoted by $s$.
 The sheaf $s^*\calO(1)$ on $C \times T$ has constant degree $d$ along each slice $C \times \{t\}$,
 so it gives a morphism $\overline{s} \colon T \lra J_{d}$.
 Let $\overline{s}\otimes\overline{f} \colon T \lra J_n$ be the composition $T \stackrel{\overline{s} \times \overline{f}}{\lra} J_d \times J_{n-d} \stackrel{\otimes}{\lra} J_n$.
 Then the sheaf $\bigl(\id_C \times \overline{s}\otimes\overline{f}\bigr)^*\calP_n$ on $C \times T$ is isomorphic to $s^*\calO(1) \otimes \calO(f)$ modulo an element of $\Pic(T)$.
 
 The canonical map $p_1^*V^{\vee} \lra s^*\calO(1)$ on $C \times T$ gives rise to 
 a quotient map
 \[
  s^*\calO(-1) \otimes \calO(f)^{\vee} \otimes p_1^*(V^{\vee} \otimes \omega_C) \lra \calO(f)^{\vee} \otimes p_1^*\omega_C
 \]
 on $C \times T$.
 Applying the right-exact functor $R^1p_{2*}$
 and the construction $\bbP^{\vee}$,
 we obtain an embedding
 \begin{align*}
  C_{n-d} \times_{J_{n-d},\overline{f}} T &= \bbP^{\vee}\bigl(R^1p_{2*}(\calO(f)^{\vee}\otimes p_1^*\omega_C)\bigr)\\ &\hlra
  \bbP^{\vee}\bigl(R^1p_{2*}\bigl(s^*\calO(-1) \otimes \calO(f)^{\vee} \otimes p_1^*(V^{\vee} \otimes \omega_C)\bigr)\bigr) = X_n(V) \times_{J_n,\overline{s}\otimes\overline{f}} T,
 \end{align*}
  by Lemma \ref{lem:coh and base chg}.
 Note also that for any coherent sheaf $\calF$ on $C \times T$ and $\calL \in \Pic(T)$, there is a canonical isomorphism $\bbP^{\vee}\bigl(R^1p_{2*}(\calF \otimes p_2^*\calL)\bigr) \cong \bbP^{\vee}(R^1p_{2*}\calF)$.
 Now we associate a $T$-valued point of $X_n(V)$ defined by the composition
 \[
  T \stackrel{(f,\id_T)}{\lra} C_{n-d} \times_{J_{n-d},\overline{f}} T \lra X_n(V) \times_{J_n,\overline{s}\otimes\overline{f}} T \lra X_{n}(V)
 \]
 to the given pair $(s,f) \in \Sect_d\bigl(\bbP(V)\bigr)(T) \times C_{n-d}(T)$.
 The construction is clearly natural in $T$ and therefore defines a morphism $\Sect_d\bigl(\bbP(V)\bigr) \times C_{n-d} \lra X_n(V)$.

 (ii)
 We must prove that the morphism $m$ in the statement induces a bijection on the sets of $k'$-valued points for any field $k'$ containing $k$.
 As explained in the proof of (i), there is a morphism $\Sect_d\bigl(\bbP(V)\bigr) \lra J_d$ defined by $s \lmt s^*\calO(1)$.
 By the construction in (i), the following diagram commutes:
 \[
 \xymatrix{
 \coprod_{d\in\bbZ}\Sect_d\bigl(\bbP(V)\bigr)\times C_{n-d} \ar[r]^(0.7){m} \ar[d] & X_n(V) \ar[d]\\
 \coprod_{d\in\bbZ}J_d \times J_{n-d} \ar[r]^(0.65){\otimes} & J_n.
 }
 \]
 Fix a point $\calL \in J_n(k')$ and consider the fiber of the above diagram at $\calL$.
 Then the fiber of the upper horizontal morphism $m$ is described as follows:
 \begin{itemize}
  \item The source is isomorphic to 
	\begin{align*}
	 &\coprod_{d \in \bbZ} \bigl\{ (s,D) \in \Sect_d\bigl(\bbP(V)\bigr)(k') \times C_{n-d}(k') \bigm | s^*\calO(1) \otimes \calO(D) \cong \calL \bigr\} \\
	\cong & \coprod_{d \in \bbZ} \Bigl[\text{ the set of pairs } (s,l) \text{ where } s \in \Sect_d\bigl(\bbP(V)\bigr)(k') \text{ and } l \in \bbP\bigl(H^0(C_{k'},\calL \otimes s^*\calO(-1))\bigr)\Bigr].
	\end{align*}
  \item The target is isomorphic to $\bbP\bigl(H^0(C_{k'}, \calL \otimes V_{k'})\bigr)$ by Lemma \ref{lem:coh and base chg}.
  \item An element $(s,l) \in \Sect_d\bigl(\bbP(V)\bigr)(k') \times \bbP\bigl(H^0(C_{k'},\calL \otimes s^*\calO(-1))\bigr)$
	is mapped to the image of $l$ under the embedding
	\[
	\bbP\bigl(H^0(C_{k'},\calL \otimes s^*\calO(-1))\bigr) \hlra \bbP\bigl(H^0(C_{k'}, \calL \otimes V_{k'})\bigr)
	\]
	which is induced by the canonical inclusion $s^*\calO(-1) \hlra V_{k'}$.
 \end{itemize}
 By the last description of the map, one can easily check that it is bijective.
 The inverse map is given as follows:
 Represent an element of $\bbP\bigl(H^0(C_{k'}, \calL \otimes V_{k'})\bigr)$ by a nonzero global section $t$ of $\calL \otimes V_{k'}$.
 Consider the divisor of zeros $\div(t)$ (i.e., the coefficient of $\div(t)$ at a point $p \in C_{k'}$ is given by the largest integer $m$ such that $t_p \in \pi_p^m(\calL \otimes V_{k'})_p$ where $\pi_p$ is a uniformizer of $C_{k'}$ at $p$).
 Then the injective morphism $\calO(\div(t)) \hlra \calL \otimes V_{k'}$ of $\calO_{C_{k'}}$-modules defined by $1 \lmt t$ is locally splitting.
 Define $d = n - \deg(\div(t))$, and let $s \in \Sect_d\bigl(\bbP(V)\bigr)(k')$ be a section such that $s^*\calO(-1) = \calL^{\vee} \otimes \calO(\div(t)) \subset V_{k'}$, and $l$ be the line in $H^0\bigl(H_{k'},\calO(\div(t))\bigr)$ generated by $1$.
 Then the pair $(s,l)$ gives a lift of $t$.
 \end{proof}

\begin{cor}
  \[
 \sum_{d \in \bbZ} \bigl[\Sect_d\bigl(\bbP(V)\bigr)\bigr][C_{n-d}] = \bigl[X_n(V)\bigr] \text{ in }\widetilde{K}, \text{and}
  \]
  \[
  \zeta(t)Z\bigl(\bbP(V),t\bigr) = \sum_{n \in \bbZ}[X_n(V)]t^n \text{ in }\widetilde{K}[[t]].
  \]
\end{cor}
\begin{proof}
  This follows from Proposition \ref{prop:computation of prod of sect and C}
  and the definition of $\widetilde{K}$.
\end{proof}

 From now we mimic Kapranov's arguments to prove the rationality and the functional equation of $\zeta(t)Z\bigl(\bbP(V),t\bigr)$.
 
\subsection{The Riemann-Roch and a vanishing result}
 In the ring $K$, $[\bbP^n] = 1 + \bbL + \cdots + \bbL^n$ for $n \geq 0$.
 From this we have the equality $[\bbP^m] - \bbL^{m-n}[\bbP^n] = [\bbP^{m-n-1}]$ for all $m,n \in \bbZ$ such that $m > n \geq 0$.
 In particular, there is a recursive relation $[\bbP^{n+1}] - \bbL[\bbP^n] = 1$.
 We also define $[\bbP^n]$ for negative $n$ in the ring $\widetilde{K}$ by this relation.
 Then the formula $[\bbP^m] - \bbL^{m-n}[\bbP^n] = [\bbP^{m-n-1}]$ holds for all $m,n \in \bbZ$ in $\widetilde{K}$.
\begin{prop}
\label{prop:motivic R-R}
 Let $V$ be a vector bundle over $C$ of rank $r$.
 \begin{enumerate}
  \item (Riemann-Roch) For all $n \in \bbZ$, 
 \[
 [X_n(V)] - \bbL^{r(n+1-g)+\deg V}[X_{2g-2-n}(V^{\vee})]
 = [\bbP^{r(n+1-g)+\deg V-1}][J].
 \]
  \item (Vanishing) $X_n(V) = \emptyset$ for enough small $n$.
\end{enumerate}
\end{prop}
\begin{proof}
 (i) We identify $J_{n}$ and $J_{2g-2-n}$ via the map $\calL \lmt \calL^{\vee} \otimes \omega_C$.
 Let $J_n = \coprod_{i}Z_i$ be a finite decomposition into locally closed subsets such that both of the sheaves $R^1p_{2*}(\calP^{\vee}_n \otimes p_1^*V^{\vee} \otimes p_1^*\omega_C)$ on $J_n$ and $R^1p_{2*}(\calP^{\vee}_{2g-2-n} \otimes p_1^*V \otimes p_1^*\omega_C)$ on $J_{2g-2-n}$ are free on each $Z_i$.
 By Riemann-Roch theorem, the difference of ranks of these sheaves is $r(n+1-g)+\deg V$.
 So we have
 \[
  [X_n(V)|_{Z_i}] - \bbL^{r(n+1-g)+\deg V}[X_{2g-2-n}(V^{\vee})|_{Z_i}]
 = [\bbP^{r(n+1-g)+\deg V-1}][Z_i]
 \]
 for each $i$.
 Summing up these equalities gives the desired result.

 (ii) This is clear from Lemma \ref{lem:coh and base chg}
 and the fact that the is a constant $n(V)$ depending only on $V$ such that $H^0(C_{k'},\calL \otimes V_{k'}) = 0$ for $\deg \calL \leq n(V)$.
\end{proof}

\subsection{Rationality and functional equations}
\begin{lem}
\label{lem:lem on power series with coef P^n}
 \begin{enumerate}
  \item Let $a,b$ be integers.
	The power series
	\[
	 g(t) = (t-1)(t-\bbL^{-a})\sum_{n \geq 0}[\bbP^{an+b}]t^n \in \calM[[t]]
	\]
	is a polynomial.
	The value of $g(t)$ at $t=\bbL^{-a}$ is $\bbL^{b-a}\bigl(1-[\bbP^{-a}]\bigr)$.
  \item The series $\sum_{n < 0}[\bbP^{an+b}]t^n$ is also a rational function of $t^{-1}$, and as an element of $\calM(t^{-1}) = \calM(t)$ we have
	\[
	 \sum_{n < 0}[\bbP^{an+b}]t^n + \sum_{n \geq 0}[\bbP^{an+b}]t^n = 0.
	\]
 \end{enumerate}
\end{lem}
\begin{proof}
 By the recursive relation $[\bbP^{an+b}] = \bbL[\bbP^{an+b-1}]+1$,
 we easily reduce the proof to the case that $b = 0$.
 Then, by a direct computation, one can show that
 \[
  (t-1)(t-\bbL^{-a})\sum_{n \geq 0}[\bbP^{an}]t^n = \bbL^{-a}\bigl(1+\bbL[\bbP^{a-2}]t\bigr),
 \]
 and the statement of the lemma follows from it easily.
\end{proof}
 
Our motivic main result is the following theorem.
\begin{thm}
\label{thm:motivic main theorem}
 The product $\zeta(t)Z\bigl(\bbP(V),t\bigr)$ is rational as a series in $\calM((t))$ with denominator $(t-1)(t-\bbL^{-r})$.
 The value of $(t-1)(t-\bbL^{-r})\zeta(t)Z\bigl(\bbP(V),t\bigr)$ at $t = \bbL^{-r}$
 is $\bbL^{-rg-1+\deg V}\bigl(1-[\bbP^{-r}]\bigr)[J]$.
 Moreover, these functions satisfy the following functional equations:
 \[
  \zeta(t)Z\bigl(\bbP(V),t\bigr) = \bbL^{\deg V+r(g-1)}t^{2g-2}\zeta(\bbL^{-r}t^{-1})Z\bigl(\bbP(V^{\vee}),\bbL^{-r}t^{-1}\bigr).
 \]
 in $\calM((t))$.
\end{thm}
\begin{proof}
 By Proposition \ref{prop:motivic R-R}, it follows that
 \[
  \zeta(t)Z\bigl(\bbP(V),t\bigr) = (\text{polynomial}) + \sum_{n \geq 0}[\bbP^{r(n+1-g)+\deg V - 1}][J]t^n.
 \]
 Then the rationality and the result on the behaviour at $t = \bbL^{-r}$ follow from Lemma \ref{lem:lem on power series with coef P^n} (i).

 Moreover, by the Riemann-Roch, we have
 \[
  \zeta(t)Z\bigl(\bbP(V),t\bigr) = 
  \sum_{n \in \bbZ}\bbL^{r(n+1-g)+\deg V}\bigl[X_{2g-2-n}(V^{\vee})\bigr]t^n + \sum_{n \in \bbZ}[\bbP^{r(n+1-g)+\deg V-1}][J]t^n.
 \]
 The second summation is zero by \ref{lem:lem on power series with coef P^n} (ii).
 Then, by substituting $m$ for $2g-2-n$, we obtain
 \begin{align*}
 \zeta(t)Z\bigl(\bbP(V),t\bigr) &= \bbL^{r(1-g)+\deg V}\sum_{m \in \bbZ}\bbL^{r(2g-2-m)}\bigl[X_m(V^{\vee})\bigr]t^{2g-2-m} \\
  &= \bbL^{r(g-1)+\deg V}t^{2g-2}\sum_{m \in \bbZ}\bigl[X_m(V^{\vee})\bigr](\bbL^{-r}t^{-1})^{m} \\
  &= \bbL^{r(g-1)+\deg V}t^{2g-2}\zeta(t)Z\bigl(\bbP(V^{\vee},\bbL^{-r}t^{-1})\bigr).
 \end{align*}
\end{proof}

In the situation where $[\bbG_m]$ is invertible,
the residue of $Z\bigl(\bbP(V),t\bigr)$ at $t = \bbL^{-r}$
is expressed in a very similar form as in Theorem \ref{thm:main theorem}.
\begin{cor}
 Let $R$ be a field and $\mu \colon \calM \lra R$ be a ring homomorphism such that $\mu(\bbL) \neq 1$.
 Let $\zeta_{\mu}(t) \in R[[t]]$ (resp. $Z_{\mu}\bigl(\bbP(V),t\bigr) \in R((t))$) be the power series obtained by evaluating each coefficient of $\zeta(t)$ (resp. $Z\bigl(\bbP(V),t\bigr)$) by $\mu$.
 Assume that $\zeta_{\mu}\bigl(\mu(\bbL)^{-r}\bigr) \neq 0$.
 Then $Z_{\mu}\bigl(\bbP(V),t\bigr)$ is a rational function of $t$ with an unique simple pole at $t = \mu(\bbL)^{-r}$ and
 \[
  \Res_{t = \mu(\bbL)^{-r}}Z_{\mu}\bigl(\bbP(V),t\bigr) = -\frac{\mu([J])\mu(\bbL)^{\deg V}}{\mu([\bbG_m])\mu(\bbL)^{rg}\zeta_{\mu}\bigl(\mu(\bbL)^{-r}\bigr)}.
 \]
\end{cor}
\begin{proof}
 This follows from Theorem \ref{thm:motivic main theorem} and the fact that
 \[
  \frac{1-\mu([\bbP^{-r}])}{\mu(\bbL)^{-r}-1} = - \frac{\mu(\bbL)}{\mu(\bbL)-1} = - \frac{\mu(\bbL)}{\mu([\bbG_m])}.
 \]
\end{proof}

\section{Hirzebruch surfaces}
The methods in \S\ref{sec:Height zeta functions} are also useful to study height zeta functions of projective bundles over a variety $X$ when the height zeta function $Z(X,s)$ of $X$ is well understood.
To demonstrate this, we study height zeta functions of Hirzebruch surfaces.

\subsection{Setting}
\label{ssect:Setting}
Fix an Arakelov vector bundle $V$ of rank $2$ over $F$
and let $\bbP^1 = \bbP(V)$ denote the associated projective line.
For an integer $e \geq 0$, let $F_e = \bbP\bigl(\calO_{\bbP^1} \otimes \calO_{\bbP^1}(e)\bigr)$ be the Hirzebruch surface and $\pi \colon F_e \lra \bbP^1$ be the canonical projection.
Let $\calO(a,b)$ denote the sheaf $\calO_{F_e}(a) \otimes \pi^*\calO_{\bbP^1}(b)$ on $F_e$, where $\calO_{F_e}(1)$ is the relative Serre's twisting sheaf of the projective bundle $\pi \colon F_e \lra \bbP^1$.
It is known that any invertible sheaf on $F_e$ is isomorphic to the sheaf $\calO(a,b)$ for some $(a,b) \in \bbZ^2$,
and it is ample iff $a > 0$ and $b > ae$ (see the section of ruled surfaces in \cite{Hartshorne}).
Fix a pair $(a,b)$ with this condition.
For a technical reason we assume that $e \geq 2$.

For a rational point $P \in F_e(F)$, the $F$-vector space $P^*\calO(a,b)$ is naturally endowed with a structure of an Arakelov line bundle.
We define the height $H(P) = H_{a,b}(P)$ as the norm $N\bigl(P^*\calO(a,b)\bigr)$.
If $Q = \pi(P)$, then $H_{a,b}(P)$ is equal to $H_{\pi^{-1}(Q)}(P)^a H_{\bbP^1}(Q)^b$
(here we see $\pi^{-1}(Q)$ as the projective space associated to the Arakelov vector bundle $Q^*\bigl(\calO_{\bbP^1}\oplus\calO_{\bbP^1}(e)\bigr)$ and consider the height function on it defined in \S\ref{ssec:Height}).
\begin{lem}
 The height zeta function $Z(F_e,s) = \sum_{P \in F_e(F)}H_{a,b}(P)^{-s}$
 satisfies
 \[
  Z(F_e,s) = \sum_{Q \in \bbP^1(F)}H_{\bbP^1}(Q)^{-bs}Z\bigl(\pi^{-1}(Q),as\bigr).
 \]
\end{lem}

\begin{proof}
 This is clear from the description of the height in the above.
\end{proof}

\subsection{Strategy}
We are going to show that $Z(F_e,s)$ defines a meromorphic function on the domain $D = \bigl\{s \in \bbC \bigm| \Re(s) > \max\{1/a,(e+2)/b\} \bigr\}$,
and its largest pole is at $s = \max\{2/a,2/(b-ea)\}$.
Note that $(e+2)/b < (e+2)/ae = 1/a + 2/ae \leq 2/a$ since $e \geq 2$, so $s = 2/a$ is always in $D$.

By the formula (\ref{eq:star}) in \S\ref{ssect:Application of Riemann-Roch}, the function $w\xi(as)Z\bigl(\pi^{-1}(Q),as\bigr)$ is the sum of the following $4$ functions:
\begin{align*}
 G_{1,Q}(s) &= - \frac{\alpha|\Delta|^{-\frac{as}{2}}}{as},\\
 G_{2,Q}(s) &= \frac{\alpha N\bigl(Q^*(\calO_{\bbP^1}\oplus\calO_{\bbP^1}(e))\bigr)|\Delta|^{-\frac{as}{2}}}{s-2} = \frac{\alpha H_{\bbP^1}(Q)^e|\Delta|^{-\frac{as}{2}}}{as-2}, \\
 G_{3,Q}(s) &= |\Delta|^{1-as}H(Q)^e\int_{E_{-}}N(L)^{as-2}\varphi\bigl(L\oplus(L\otimes Q^*\calO_{\bbP^1}(-e))\bigr)dL, \\
 G_{4,Q}(s) &= \int_{E_{-}}N(L)^{-as}\varphi\bigl(L\oplus(L\otimes Q^*\calO_{\bbP^1}(e))\bigr)dL.
\end{align*}
From now we try to find the largest pole of the series $F_i(s) = \sum_{Q \in \bbP^1(F)}H_{\bbP^1}(Q)^{-bs}G_{i,Q}(s)$ in the domain $D$ for each $i \in \{1,2,3,4\}$.

\subsection{First three parts}
The first part $F_1(s)$ is
\[
 F_1(s) = - \frac{\alpha|\Delta|^{-\frac{as}{2}}}{as}\sum_{Q \in \bbP^1(F)}H_{\bbP^1}(Q)^{-bs} = - \frac{\alpha|\Delta|^{-\frac{as}{2}}}{as}Z\bigl(\bbP^1,bs\bigr).
\]
By Theorem \ref{thm:main theorem}, the series $Z(\bbP^1,bs)$ is holomorphic for $s > 2/b$.
Then $F_1(s)$ is holomorphic in $D$ because 
we have $2/b < 1/a$ from $b > ae \geq 2a$.

The second part is
\[
 F_2(s) = \frac{\alpha |\Delta|^{-\frac{as}{2}}}{as-2}\sum_{Q \in \bbP^1(F)}H_{\bbP^1}(Q)^{-(bs-e)} = \frac{\alpha |\Delta|^{-\frac{as}{2}}}{as-2}Z\bigl(\bbP^1,bs-e\bigr),
\]
Therefore $F_2(s)$ has its largest pole at $s=2/a$ with order $1$ and residue $\alpha a^{-1} |\Delta|^{-1} Z\bigl(\bbP^1, 2b/a - e\bigr)$.

The third part is
\[
 F_3(s) = |\Delta|^{1-as}\sum_{Q \in \bbP^1(F)}H_{\bbP^1}(Q)^{-(bs-e)}\int_{E_{-}}N(L)^{as-2}\varphi\bigl(L\oplus(L\otimes Q^*\calO_{\bbP^1}(-e))\bigr)dL.
\]
Since the line bundle $Q^*\calO_{\bbP^1}(-e)$ is embedded into $V^{\otimes e}$, we have an estimate
\[
 \int_{E_{-}}\bigl|N(L)^{as-2}\bigr|\varphi\bigl(L\oplus(L\otimes Q^*\calO_{\bbP^1}(-e))\bigr)dL \leq \int_{E_{-}}\bigl|N(L)^{as-2}\bigr|\varphi\bigl(L\oplus(L\otimes V^{\otimes e})\bigr)dL
\]
where the latter integral converges for any $s \in \bbC$ by Theorem \ref{prop:R-R and vanishing} (ii) and their values are bounded with respect to $Q$.
Therefore the third part has no pole in the domain $D$.

\subsection{The fourth part}
In general, for two Arakelov vector bundles $V,W$ we have $\varphi(V\oplus W) = \varphi(V) + \varphi(W) + \varphi(V)\varphi(W)$ which follows from $\#H^0(V\oplus W) = \#H^0(V)\#H^0(W)$.
Then the fourth part $F_4(s)$ is divided into the sum $F_4^{(1)}(s) + F_4^{(2)}(s) + F_4^{(3)}(s)$ where
\begin{align*}
 F_4^{(1)}(s) &= \sum_{Q \in \bbP^1(F)}H_{\bbP^1}(Q)^{-bs}\int_{E_{-}}N(L)^{-as}\varphi(L)dL, \\
 F_4^{(2)}(s) &= \sum_{Q \in \bbP^1(F)}H_{\bbP^1}(Q)^{-bs}\int_{E_{-}}N(L)^{-as}\varphi\bigl(L\otimes Q^*\calO_{\bbP^1}(e)\bigr)dL, \\
  F_4^{(3)}(s) &= \sum_{Q \in \bbP^1(F)}H_{\bbP^1}(Q)^{-bs}\int_{E_{-}}N(L)^{-as}\varphi\bigl(L)\varphi\bigl(L\otimes Q^*\calO_{\bbP^1}(e)\bigr)dL.
\end{align*}

\subsubsection*{I}
The series $F^{(1)}_4(s)$ is the product of $Z\bigl(\bbP^1,bs\bigr)$ with a holomorphic function.
In particular, it has no pole in the domain $D$.

\subsubsection*{II}
In the series $F^{(2)}_4(s)$,
replacing $L \otimes Q^*\calO_{\bbP^1}(e)$ to $L$ yields
\begin{align*}
 F_4^{(2)}(s) &= \sum_{Q \in \bbP^1(F)}H_{\bbP^1}(Q)^{-(b-ae)s}\int_{N(L)\leq\sqrt{|\Delta|}H_{\bbP^1}(Q)^e}N(L)^{-as}\varphi(L)dL \\
 &= wZ\bigl(\bbP^1,(b-ae)s\bigr)\xi(as) - F^{(2)'}_4(s)
\end{align*}
where
\[
 F^{(2)'}_4(s) = \sum_{Q \in \bbP^1(F)}H_{\bbP^1}(Q)^{-(b-ae)s}\int_{N(L)\geq\sqrt{|\Delta|}H_{\bbP^1}(Q)^e}N(L)^{-as}\varphi(L)dL. 
\]
By the Riemann-Roch therem, $\varphi(L) = \varphi(L^{\vee}\otimes\omega)N(L)|\Delta|^{-\frac{1}{2}} + N(L)|\Delta|^{-\frac{1}{2}} - 1$.
Then $F^{(2)'}_4$ is divided as follows:
\begin{align*}
 F^{(2)'}_4(s)
 &= |\Delta|^{-\frac{1}{2}}\sum_{Q \in \bbP^1(F)}H_{\bbP^1}(Q)^{-(b-ae)s}\int_{N(L)\geq\sqrt{|\Delta|}H_{\bbP^1}(Q)^e}N(L)^{-as+1}\varphi(L^{\vee}\otimes\omega)dL \\
 &+ |\Delta|^{-\frac{1}{2}}\sum_{Q \in \bbP^1(F)}H_{\bbP^1}(Q)^{-(b-ae)s}\int_{N(L)\geq\sqrt{|\Delta|}H_{\bbP^1}(Q)^e}N(L)^{-as+1}dL \\
 &- \sum_{Q \in \bbP^1(F)}H_{\bbP^1}(Q)^{-(b-ae)s}\int_{N(L)\geq\sqrt{|\Delta|}H_{\bbP^1}(Q)^e}N(L)^{-as}dL.
\end{align*}
The integrals in the second and third term can be explicitely computed by the formula \\ $\int_{\Pic(F)}f(\deg L)dL = \alpha \int_{\bbR}f(t)dt$ in \S\ref{subs:Topologies and measures}.
The results are
\begin{align*}
 \int_{N(L)\geq\sqrt{|\Delta|}H_{\bbP^1}(Q)^e}N(L)^{-as+1}dL &= \frac{\alpha \sqrt{|\Delta|}^{1-as}H_{\bbP^1}(Q)^{e(1-as)}}{as-1}\quad\text{for }\Re(s)>\frac{1}{a}, \\
 \int_{N(L)\geq\sqrt{|\Delta|}H_{\bbP^1}(Q)^e}N(L)^{-as}dL &=
 \frac{\alpha\sqrt{|\Delta|}^{-as}H_{\bbP^1}(Q)^{-eas}}{as}\quad\text{for }\Re(s)>0.
\end{align*}
From this it is easily checked that the largest pole of the second (resp. third) term is at $s = (e+2)/b$ (resp. $s = 2/b$).

The first term in the above expression of $F^{(2)'}_4(s)$ is holomorphic for $\Re(s) > 1/a$.
In fact, since $N(L^{\vee}\otimes\omega) \leq \sqrt{|\Delta|}H_{\bbP^1}(Q)^{-e}$, we have $\varphi(L^{\vee}\otimes\omega) \leq C_1 \exp \bigl(-C_2 H_{\bbP^1}(Q)^{C_3}\bigr)$ for some $C_1,C_2,C_3 > 0$ by Proposition \ref{prop:R-R and vanishing} (ii).
Note also that the value of the integral $\int_{N(L)\geq\sqrt{|\Delta|}H_{\bbP^1}(Q)^e}N(L)^{-as+1}dL$ is bounded with respect to $Q$ (for $\Re(s) > 1/a$).
From these facts one can easily see that the first term converges absolutely for $\Re(s) > 1/a$.

We have seen that the pole of the series $F^{(2)}_4(s)$ in $D$ arises from the term $wZ\bigl(\bbP^1,(b-ae)s\bigr)\xi(as)$.
Its largest pole is 
at $s = 2/(b-ea)$ (when $2/(b-ea) \in D$).
It is simple and the residue is given by
\[
 \lim_{s \to 2/(b-ae)}w\biggl(s-\frac{2}{b-ae}\biggr)Z\bigl(\bbP^1,(b-ae)s\bigr)\xi(as) = \frac{\xi(\frac{2a}{b-ae})}{b-ae} \cdot \frac{\alpha N(V)}{|\Delta| \xi(2)}.
\]

\subsubsection*{III}
Now we consider the last part $F_4^{(3)}(s)$.
Divide $E_{-}$ into $E'_{-}(Q) \cup E''_{-}(Q)$ where 
$E_{-}'(Q) = \{L \in E_{-} \mid N(L)H(Q)^{e/2} \leq \sqrt{|\Delta|}\}$ and
$E_{-}''(Q) = \{L \in E_{-} \mid N(L)H(Q)^{e/2} \geq \sqrt{|\Delta|} \}$.
Associated to this, we divide the series $F^{(3)}_4(s)$ into $F^{(3)'}_4(s) + F^{(3)''}_4(s)$ where
\begin{align*}
 F_4^{(3)'}(s) &= \sum_{Q \in \bbP^1(F)}H_{\bbP^1}(Q)^{-bs}\int_{E'_{-}(Q)}N(L)^{-as}\varphi\bigl(L)\varphi\bigl(L\otimes Q^*\calO_{\bbP^1}(e)\bigr)dL, \\
 F_4^{(3)''}(s) &= \sum_{Q \in \bbP^1(F)}H_{\bbP^1}(Q)^{-bs}\int_{E''_{-}(Q)}N(L)^{-as}\varphi\bigl(L)\varphi\bigl(L\otimes Q^*\calO_{\bbP^1}(e)\bigr)dL.
\end{align*}

First we consider $F^{(3)'}_4(s)$, so assume that $L \in E_{-}'(Q)$.
Then $N(L) \leq \sqrt{|\Delta|} H(Q)^{-e/2}$,
and by Proposition \ref{prop:R-R and vanishing} (ii),
$\sqrt{\varphi(L)} \leq C_1\exp(-C_2H(Q)^{C_3})$ for some $C_1,C_2,C_3 > 0$.
Moreover, we have $N\bigl(L \otimes Q^*\calO_{\bbP^1}(e)\bigr) \leq \sqrt{|\Delta|}H(Q)^{e/2}$ and $\varphi\bigl(L \otimes Q^*\calO_{\bbP^1}(e)\bigr) \leq C_4H(Q)^{e/2}$ for some $C_4 > 0$ by Proposition \ref{prop:R-R and vanishing} (iii).
Therefore, we have an estimate
\begin{align*}
 &\Bigl|H_{\bbP^1}(Q)^{-bs}N(L)^{-as}\varphi(L)\varphi\bigl(L\otimes Q^*\calO_{\bbP^1}(e)\bigr)\Bigr| \\
 &\leq C_5\Bigl|H_{\bbP^1}(Q)^{-bs+e/2}\exp\bigl(-C_2H_{\bbP^1}(Q)^{C_3}\bigr)\Bigr|\bigl|N(L)^{-as}\sqrt{\varphi(L)}\bigr|.
\end{align*}
for some $C_5 > 0$.
Since $E_{-}'(Q) \subset E_{-}$, the integral $\int_{E_{-}'(Q)}N(L)^{-as}\sqrt{\varphi(L)}dL$ converges for any $s \in \bbC$ and their values are bounded with respect to $Q$.
So the series $F^{(3)'}_4(s)$ converges for any $s \in \bbC$.

Next we work with $F^{(3)''}_4(s)$.
By the Riemann-Roch theorem,
\begin{align*}
 F^{(3)''}_4(s) &= |\Delta|^{-\frac{1}{2}}\sum_{Q \in \bbP^1(F)}H_{\bbP^1}(Q)^{-(bs-e)}\int_{E''_{-}(Q)}N(L)^{-(as-1)}\varphi(L)\varphi\bigl(L^{\vee}\otimes Q^*\calO_{\bbP^1}(-e)\otimes\omega\bigr)dL \\
 &+ |\Delta|^{-\frac{1}{2}}\sum_{Q \in \bbP^1(F)}H_{\bbP^1}(Q)^{-(bs-e)}\int_{E''_{-}(Q)}N(L)^{-(as-1)}\varphi(L)dL \\
 &- \sum_{Q \in \bbP^1(F)}H_{\bbP^1}(Q)^{-bs}\int_{E''_{-}(Q)}N(L)^{-as}\varphi(L)dL.
\end{align*}
All of three integrals in the above expression converges for any $s \in \bbC$ and its value is bounded with respect to $Q$.
So the series $F^{(3)''}_4(s)$ is holomorphic in the domain $D$.

\subsection{Conclusion}
We have proved the following theorem.
\begin{thm}
 \label{thm:Hirzebruch main theorem}
 The height zeta function $Z(F_e, s) = \sum_{P \in F_e(F)}H_{a,b}(P)^{-s}$ defines a meromorphic function in the domain $D = \bigl\{s \in \bbC \bigm| \Re(s) > \max\{1/a,(e+2)/b\} \bigr\}$.
 It has a simple pole at $s = 2/a$ with residue
 \[
  \frac{\alpha Z(\bbP^1,2b/a-e)}{a w |\Delta| \xi(2)},
 \]
 and another simple pole at $s = 2/(b-ea)$ with residue
 \[
  \frac{\alpha N(V)}{(b-ae)w|\Delta|\xi(2)}.
 \]
 when $2/(b-ea) \in D$.
 There are no other poles in $D$.
\end{thm}

\subsection{Discussions}
\label{ssect:Discussions}
Consider the minimal section of $\pi \colon F_e = \bbP\bigl(\calO\oplus\calO(e)\bigr) \lra \bbP^1$ which corresponds to a line-subbundle $\calO(e) \subset \calO\oplus\calO(e)$.
Let $C \subset F_e$ be its image.
Then the above result shows that the points on $C$ are dominant in the distribution of rational points of $F_e$ when $2/(b-ea) > 2/a \iff b < (e+1)a$.
In fact, the restriction $\calO(a,b)|_C$ is isomorphic to $\calO_{\bbP^1}(b-ea)$ if we identify $C$ and $\bbP^1$ via the projection.
Therefore by Theorem \ref{thm:main theorem} the height zeta function of $C$ with respect to $\calO(a,b)|_C$ is equal to $Z\bigl(\bbP^1,(b-ea)s\bigr)$
and the principal part of its largest pole
is equal to that of $Z(F_e,s)$ when $b<(e+1)a$.

Before considering the case $b \geq (e+1)a$,
we recall the conjecture of Batyrev and Manin (Conjecture B of \cite{BatyrevManin})
on the abscissa of convergence of height zeta functions of Fano varieties:
Let $X$ be a Fano variety over $F$, and $L$ be an ample line bundle on $X$.
Let $\alpha(L) = \inf\{A \in \bbR \mid AL + K_X \in \Lambda_{\text{eff}}\}$
where $\Lambda_{\text{eff}} \subset \Pic(X)_{\bbR}$ denotes the effective cone of $X$.
Then for any sufficiently big finite extension $F' \supset F$ and for any sufficiently small Zariski open subset $U \subset X$,
the abscissa of convergence of $Z(U_{F'},H_{L};s)$ is equal to $\alpha(L)$.

Now return to our case $X=F_e$ and $b \geq (e+1)a$.
It is known that the canonical sheaf $K_{F_e}$ is isomorphic to $\calO(-2,-2-e)$ (see the section of ruled surfaces in \cite{Hartshorne}).
Also recall that the effective cone $\Lambda_{\text{eff}} \subset \Pic(F_e)_{\bbR} \cong \bbR^2$ is $\{(a,b) \in \bbR^2 \mid a \geq 0, b \geq ae \}$.
Therefore the number $2/a$ (i.e., the abscissa of convergence of the height zeta function $Z(F_e,s)$) is equal to
$\inf\{A \in \bbR \mid A\calO(a,b) + K_{F_e} \in \Lambda_{\text{eff}}\}$.

As a whole, Batyrev and Manin's description of the abscissa of convergence is valid for Hirzebruch surfaces $F_e \ (e \geq 2)$ with respect to ample line bundles $L=\calO(a,b) \ \bigl(b \geq (e+1)a\bigr)$
without shrinking $F_e$ to a small open subset.

\section{Tauberian theorem}
We include here the statement of Tauberian theorem
used to deduce the asymptotic behaviour of the counting function $n(X,H_L;B)$ from properties of the height zeta function $Z(X,H_L;s)$.

\begin{thm}[Tauberian theorem]
\label{thm:Tauberian}
 Let $\{\lambda_n\}_{n \in \bbN}$ be a non-decreasing sequence of positive real numbers.
 Let $a \in \bbR$ be a positive real number and $b \in \bbN$ a positive integer.
 Assume that the Dirichlet series $Z(s) = \sum_{n \in \bbN}\lambda_n^{-s}$ admits a representation
 \[
  Z(s) = \frac{g(s)}{(s-a)^b} + h(s)
 \]
 on some open neighbourhood of $\bigl\{ s \in \bbC \bigm| \Re(s) \geq a \bigr\}$, where $g(s)$ and $h(s)$ are holomorphic functions such that $g(a) \neq 0$.
 Then
 \[
  N(B) = \#\{n \in \bbN \mid \lambda_n \leq B\} = \frac{g(a)}{a(b-1)!}B^{a}(\log B)^{b-1}\bigl(1+o(1)\bigr) \quad \text{as } B \to \infty.
 \]
\end{thm}

\begin{proof}
 Apply Theorem III in \cite{Delange} to $\alpha(t) = N(e^t)$.
 Note that $f(s)$ in \cite{Delange} is equal to $Z(s)/s$.
\end{proof}

\bibliographystyle{plain}
\bibliography{references}
\end{document}